\documentclass[12pt,fleqn]{amsart}

\usepackage[utf8]{inputenc}
\usepackage[T1]{fontenc}

\usepackage{a4wide}
\usepackage{scrextend}

\usepackage{amssymb, amsmath, amsthm, mathtools}
\usepackage{mathrsfs} 
\usepackage[mathscr]{euscript}
\usepackage{dsfont} 
\usepackage{graphicx} 
\usepackage{enumerate} 
\usepackage[shortlabels]{enumitem} 

\usepackage{lmodern}
\usepackage{dsfont}
\usepackage{accents}
\usepackage{bbm}

\usepackage[table, svgnames, x11names]{xcolor}
\usepackage[colorinlistoftodos, prependcaption, textsize=tiny, textwidth=3.5cm]{todonotes}
\presetkeys{todonotes}{fancyline}{}

\newcommand{\jarekin}[1]{\todo[inline, linecolor=Green, backgroundcolor=Green!25, bordercolor=Green, caption={Jarek}]{%
    \begin{minipage}{\textwidth-30pt}#1\end{minipage}}}

\newcommand{\IR}{\mathbb{R}}

\newcommand{\IQ}{\mathbb{Q}}
\newcommand{\N}{\mathbb{N}}
\newcommand{\IN}{\mathbb{N}}

\newcommand{\FF}{\mathscr{F}}

\renewcommand{\leq}{\leqslant}
\renewcommand{\geq}{\geqslant}
\newcommand{\mc}{\mathcal}
\newcommand{\on}{\operatorname}

\newcommand{\ve}{\mathbbm}

\newcommand{\lnn}[1]{\left\vert\kern-0.25ex\left\vert\kern-0.25ex\left\vert #1 
    \right\vert\kern-0.25ex\right\vert\kern-0.25ex\right\vert} 

\newcounter{maintheorem}

\newtheorem{mainth}[maintheorem]{Theorem}
\newcounter{maincorollary}

\newtheorem{theorem}{Theorem}[section]
\newtheorem{lemma}[theorem]{Lemma}
\newtheorem{proposition}[theorem]{Proposition}
\newtheorem{corollary}[theorem]{Corollary}
\theoremstyle{definition}
\newtheorem{definition}[theorem]{Definition}
\newtheorem{question}[theorem]{Question}
\theoremstyle{remark}
\newtheorem{remark}{Remark}

\title[Filter $\pi$-bases]{The $\pi$-property of a Banach space along a filter}
\author[T.~Kania]{Tomasz Kania}
\address[T.~Kania]{Mathematical Institute\\Czech Academy of Sciences\\\v Zitn\'a 25 \\115 67 Praha 1\\Czech Republic  and  Institute of Mathematics and Computer Science\\ Jagiellonian University\\ {\L}ojasiewicza 6, 30-348 Krak\'{o}w, Poland
}
\email{kania@math.cas.cz, tomasz.marcin.kania@gmail.com}

\author[J.~Swaczyna]{Jaros{\l}aw Swaczyna}
\address[J.~Swaczyna]{Institute of Mathematics, {\L}\'od\'z University of Technology, Aleje Politechniki 8,  93-590 {\L}\'od\'z, Poland}
\email{jaroslaw.swaczyna@p.lodz.pl}
\thanks{Funding received from NCN project SONATA BIS 13 No. 2023/50/E/ST1/00067 is acknowledged with thanks. \newline Institute of Mathematics, Czech Academy of Sciences; RVO: 67985840.}
\date{\today}

\begin{document}
\maketitle
\begin{abstract}
    We examine the analyticity of the class of separable Banach spaces possessing the $\pi$-property, defined in terms of convergence along a filter. Our results establish that this class is $\Sigma^1_3$ whenever the underlying filter is analytic (as a subset of the Cantor set $\Delta$). Furthermore, we demonstrate that if the filter is countably generated, the class of such spaces is $\Sigma^1_2$ with respect to any admissible Polish topology on the family of closed subspaces of $C(\Delta)$.
\end{abstract}

\section{Introduction}
A Banach space is said to have the \(\pi\)-property if there exists a uniformly bounded net of finite-rank projections acting on the space that converges to the identity map in the strong operator topology, \emph{i.e.}, pointwise. For separable Banach spaces, the \(\pi\)-property is strongly connected to the existence of a finite-dimensional decomposition (FDD): the presence of an FDD ensures the \(\pi\)-property, yet whether the converse holds remains an open question (\cite[p.~3]{HajekMedina}). Notably, the stronger metric \(\pi\)-property does imply the existence of an FDD.

To further explore this relationship, we propose extending the framework to encompass a broader class of spaces. Specifically, we replace convergence in the strong operator topology with convergence along a filter, while maintaining the setting of the strong operator topology. This generalisation not only sheds new light on the \(\pi\)-property but also introduces a concept of intrinsic interest. To lay the groundwork for this study, we formally present the central definition underpinning this note that is expressed in terms of filter convergence. (Let \(\mathcal{F}\) be a filter on \(\N\), and let \(X\) be a Banach space. A sequence \((x_n)_{n=1}^\infty\) in \(X\) is said to converge to \(x \in X\) along \(\mathcal{F}\) if 
\[
    \lim_{n, \mathcal{F}} \|x_n - x\| = 0,
\]
meaning that for every \(\varepsilon > 0\), there exists \(A \in \mathcal{F}\) such that \(\|x_n - x\| < \varepsilon\) for all \(n \in A\).)

\begin{definition}[\(\mathcal{F}\)-\(\pi\)-property]\label{maindef}
Let \(\mathcal{F}\) be a filter on \(\mathbb{N}\). A (separable) Banach space \(X\) has the \(\mathcal{F}\)-\(\pi\)-\emph{property} if there exists a sequence of finite-rank projections \((P_n)_{n=1}^\infty\) on \(X\) such that, for every \(x \in X\),
\[
    \lim_{n, \mathcal{F}} \|P_n x - x\| = 0.
\]
In this case, the sequence \((P_n)_{n=1}^\infty\) is called an \(\mathcal{F}\)-\(\pi\)-\emph{basis} of \(X\). An \(\mathcal{F}\)-\(\pi\)-basis with respect to the Fréchet filter is simply referred to as a \(\pi\)-basis.
\end{definition}

For the sake of avoiding pathologies, we have stated the definition for filters on $\mathbb N$ thereby limiting ourselves to separable Banach spaces. Hereinafter, all filters considered contain the Fréchet filter, that is the filter of cofinite subsets of $\mathbb N$. It is readily seen that for the Fréchet filter, Definition~\ref{maindef} renders the familiar $\pi$-property. 

Banach spaces that have filter bases having continuous basis projections constitute paradigmatic examples of spaces with the $\FF$-$\pi$-property. Indeed, suppose that $(e_n)_{n=1}^\infty$ is an $\FF$-basis for a Banach space $X$. This means that for every $x\in X$ there exists a unique sequence of scalars $(e_n^*(x))_{n=1}^\infty$ such that
\begin{itemize}
    \item for every $n\in \mathbb N$, the assignment $y\mapsto e_n^*(y)$ is continuous,
    \item $\lim\limits_{n, \FF}\sum_{j=1}^n e_n^*(x)e_n =x$.
\end{itemize}
The finite-rank operators $(P_n)_{n=1}^\infty$ given by $P_n x = \sum_{j=1}^n e_n^*(x) e_n$ (basis projections) witness the $\FF$-$\pi$-property of $X$. \smallskip

Kochanek \cite{Kochanek} demonstrated that if \(\mathcal{F}\) is countably generated (or more generally, generated by fewer than \(\mathfrak{p}\) sets, where \(\mathfrak{p}\) denotes the pseudo-intersection number), that is, there exists a countable family (or a family of cardinality less than \(\mathfrak{p}\)) \(\mathscr{B} \subset \mathcal{F}\) such that for every \( A \in \mathcal{F} \), there exists \( B \in \mathscr{B} \) with \( B \subseteq A \), then the associated projections \((P_n)_{n=1}^\infty\) are uniformly bounded on some set in the filter, and the continuity assumption in the definition of an \(\mathcal{F}\)-basis becomes redundant. 

Moreover, the present authors have recently shown that the continuity assumption in the definition of an \(\mathcal{F}\)-basis is redundant for projective filters under proper set-theoretic assumptions (\cite{KaniaSwaczyna}) and even for analytic filters in ZFC (jointly with de Rancourt \cite{RKS}). 

The proof techniques from \cite{Kochanek} inspired Avil\'es \emph{et al.}~\cite{avilesetal} to identify the class of filters possibly generated by \(\mathfrak{c}\)-sized sets for which Kochanek's argument remains valid. A filter \(\mathcal{F}\) is termed \emph{Baire} if, for every complete metric space \( X \) and every family \(\{X_A \colon A \in \mathcal{F}\}\) of nowhere-dense subsets of \( X \) satisfying \( X_A \subseteq X_B \) whenever \( A, B \in \mathcal{F} \) and \( A \subseteq B \), it holds that \(\bigcup_{A \in \mathcal{F}} X_A \neq X\). 

Under Martin's Axiom, the classes of filters that are Baire with respect to all complete metric spaces and filters generated by fewer than continuum-many sets coincide. However, there exist models of ZFC that admit filters generated by exactly \(\mathfrak{c}\)-sized sets that are Baire (\cite[Section 5]{avilesetal}). On the other hand, analytic filters that are not countably generated fail to be Baire with respect to some Polish space (\cite[Proposition 4.1]{avilesetal}), which constrains the range of spaces with the \(\mathcal{F}\)-\(\mathcal{\pi}\)-property derived from \(\mathcal{F}\)-bases. \smallskip

An \emph{FDD} (finite-dimensional decomposition) for a Banach space \(X\) is a sequence of finite-dimensional subspaces \((F_n)_{n=1}^\infty\) such that every \(x \in X\) admits a unique representation
\[
    x = \sum_{n=1}^\infty x_n, \quad \text{with} \quad x_n \in F_n \ \text{for all} \ n \in \mathbb{N}.
\]
Let us formulate the following question, which encapsulates our motivation for investigating spaces with the \(\mathcal{F}\)-\(\mathcal{\pi}\)-property.

\begin{question}\label{q:1}
    Suppose that a separable Banach space has the \(\mathcal{F}\)-\(\mathcal{\pi}\)-property with respect to a certain filter \(\FF\). Does \( X \) have a finite-dimensional decomposition (FDD)?
\end{question}

The purpose of this note is to evaluate the complexity of spaces with the \(\mathcal{F}\)-\(\mathcal{\pi}\)-property within the framework of analytic filters. This investigation is inspired by a similar result of Ghawadrah \cite{GhaFund}, which concerns spaces with the usual \(\pi\)-property.

To state our result, we introduce the set \(\mathsf{SB}_{\mathcal{F}\text{-}\pi}\), a subset of \(\mathsf{SB}\) (formally defined in the following section) consisting of Banach spaces that exhibit the \(\mathcal{F}\)-\(\mathcal{\pi}\)-property.

\begin{mainth}\label{Th:A}Let $\FF$ be an analytic (or $\Sigma^1_{s}$) filter on $\mathbb N$. Then:
\begin{itemize}
    \item the set $\mathsf{SB}_{\mathcal{F}\text{-}\pi}$ is $\Sigma_3^1$ ($\Sigma_{s+2}^1$) in $\mathsf{SB}$.
    \item if $\FF$ is countably generated, the set $\mathsf{SB}_{\mathcal{F}\text{-}\pi}$ is $\Sigma^1_2$ in $\mathsf{SB}$.
\end{itemize}
\end{mainth}

For the Fréchet filter $\mathcal{F}_{{\rm Fr}}$, that is, in the case of the usual \(\pi\)-property, Ghawadrah~\cite{GhaFund} proved that $\mathsf{SB}_{\mathcal{F}_{{\rm Fr}}\text{-}\pi}$ is $\Sigma^0_6$. Even though neither Ghawadrah's identification of the Borel class of $\mathsf{SB}_{\mathcal{F}_{{\rm Fr}}\text{-}\pi}$ nor our identification of $\mathsf{SB}_{\mathcal{F}\text{-}\pi}$ is proved to be optimal, the increase in Borel complexity by one appears to be inevitable. This suggests a negative answer to Question~\ref{q:1}, should the $\pi$-property be equivalent to the existence of an FDD for separable Banach spaces (see~\cite[Theorem 6.4]{Casazza}).

Note also that the family $\mathsf{SB}_{\mathcal{F}\text{-}\pi}$ may depend on the particular choice of the filter $\mathcal{F}$. The following theorem asserts that a certain reduction to an analytic filter is possible.
\begin{mainth}\label{Th:B}
Let $\mathcal{F}$ be a filter on $\mathbb{N}$. Suppose that $X$ is a Banach space with the $\mathcal{F}$-$\pi$-property, and let $(P_n)_{n=1}^\infty$ witness this property. Then there exists an analytic filter $\mathcal{F}' \subset \mathcal{F}$ such that $(P_n)_{n=1}^\infty$ also witnesses the $\mathcal{F}'$-$\pi$-property of $X$.
\end{mainth}
It is worth noting that Theorem~\ref{Th:B} does not render the consideration of higher projective (and beyond) complexities superfluous, as we do not assert that the choice of $\mathcal{F}'$ can be made uniformly for all $X \in \mathsf{SB}$.

\begin{mainth}\label{Th:C}
Let $\mathsf{SB}_{\operatorname{Filter-}\pi}$ denote the class of all $X \in \mathsf{SB}$ which possess the $\mathcal{F}$-$\pi$-property for some filter $\mathcal{F}$. Then $\mathsf{SB}_{\operatorname{Filter-}\pi}$ is $\Sigma^1_3$ within $\mathsf{SB}$.
\end{mainth}
\section{Preliminaries}
Our notation concerning Banach spaces is standard. We work with Banach spaces over real or complex scalar field. All operators are assumed to be bounded and linear. A~\emph{projection} is an idempotent operator. We denote the identity operator on a normed space $F$ by $I_F$.\smallskip 

By the Banach--Mazur theorem, every separable Banach space embeds isometrically into $C(\Delta)$, the space of continuous functions on the Cantor set. The space $\mathsf{SB}$ of all closed (linear) subspaces of $C(\Delta)$ is a standard Borel space when furnished with the Effros--Borel $\sigma$-algebra that is generated by the sets $E^+(U) = \{X\in \mathsf{SB}\colon X\cap U\neq \varnothing\}$, where $U$ is a~non-empty, open subset of $C(\Delta)$. There are numerous Polish topologies $\mathsf{SB}$ whose Borel $\sigma$-algebra coincides with the Effros--Borel $\sigma$-algebra, albeit no canonical one.\smallskip

Following Godefroy and Saint-Raymond \cite{GoSR}, we say that a Polish topology $\tau$ on $\mathsf{SB}$ is \emph{admissible} if it satisfies the following conditions:
\begin{itemize}
    \item For every open set $U \subseteq C(\Delta)$, the set $E^+(U)$ belongs to $\tau$;
    \item There exists a subbasis $\mathcal{B}$ of $\tau$ such that every $U \in \mathcal{B}$ can be expressed as a countable union of sets of the form $E^+(U) \setminus E^+(V)$, where $U$ and $V$ are open subsets of $C(\Delta)$.
\end{itemize}
A natural example of an admissible topology is the Wijsman topology, where $X_n \to X$ in $\mathsf{SB}$ if and only if $d(x, X_n) \to d(x, X)$ for every $x \in C(\Delta)$. Recall that a subset of a Polish space is called \emph{analytic} if it is the continuous image of a Borel set in a Polish space.

We will also make use of \cite[Theorem 4.1]{GoSR}. For a detailed study of the descriptive set-theoretic complexity of isometry classes, we refer the reader to \cite{Cuth1, Cuth2} that build upon the results from \cite{GoSR}. Note that since all codings used in those papers agree on projective classes, our results remains true in the context of all those papers.

\begin{theorem}\label{thm:continuous_selections}
There exists a sequence of continuous maps \((g_i)_{i=1}^\infty\), where \(g_i \colon \mathsf{SB} \to C(\Delta)\), such that for every \(X \in \mathsf{SB}\),
\[
    X = \overline{\{g_i(X)\colon i \in \mathbb{N}\}}.
\]
\end{theorem}
As a consequence of the above Theorem we get the following.
\begin{corollary}\label{cor:borel_selections}
    There exists a sequence of Borel maps \((f_i)_{i=1}^\infty\), where \(f_i \colon \mathsf{SB} \to C(\Delta)\), such that for every \(X \in \mathsf{SB}\),
\[
    X = \overline{\{f_i(X)\colon i \in \mathbb{N}\}}
\] and for every infinite-dimensional $X$ vectors $f_i(X)$ are linearly independent.
\end{corollary}
\begin{proof}
    Consider the functions $g_i$ provided by Theorem~\ref{thm:continuous_selections}, and let $(U_n)_{n=1}^\infty$ be a countable base for $C(\Delta)$. By \cite[Corollary 4.2]{GoSR}, the set of finite-dimensional spaces is Borel in $\mathsf{SB}$. Consequently, we may define the functions $f_i$ separately for finite- and infinite-dimensional spaces without violating Borel measurability.

    For finite-dimensional spaces, we set $f_i(X) = g_i(X)$ whenever $\dim X$ is finite. For infinite-dimensional spaces $X \in \mathsf{SB}$, we define the values $f_i(X)$ inductively. Fix $n \in \mathbb{N}$ and assume that $f_i(X)$ has already been defined for all $i < n$. We then define 
    \[
        m_n = \min \left\{ j \in \mathbb{N} \colon g_j(X) \in U_n \wedge g_j(X) \text{ is independent of } \{ f_i(X) \colon i < n \} \right\},
    \]
    whenever $U_n \cap X \neq \emptyset$, and
    \[
        m_n = \min \left\{ j \in \mathbb{N} \colon g_j(X) \text{ is independent of } \{ f_i(X) \colon i < n \} \right\},
    \]
    if $U_n \cap X = \emptyset$. The assumptions on the dimension of $X$ and the density of the set $\{ g_i(X) \colon i \in \mathbb{N} \}$ in $X$ ensure that this construction is well-defined.

    It is clear that the vectors $f_n(X)$ remain linearly independent whenever $\dim X = \mathfrak{c}$. The remaining claims follow straightforwardly. In particular, for every $n \in \mathbb{N}$, there exists a~partition of $\mathsf{SB}$ into Borel sets $(B_i)_{i=1}^\infty$ such that for each $i \in \mathbb{N}$, there exists some $j \in \mathbb{N}$ satisfying $f_n |_{B_i} = g_j |_{B_i}$.
\end{proof}
 
\subsection{Finite-dimensional structures}
Let us begin with the observation that uniformly bounded \(\mathcal{F}\)-\(\pi\)-bases closely resemble classical \(\pi\)-bases.

\begin{proposition}
Let \(X\) be a separable Banach space, and let \(\mathcal{F}\) be a filter on \(\mathbb{N}\) containing the Fréchet filter. Suppose that \((P_n)_{n=1}^\infty\) is an \(\mathcal{F}\)-\(\pi\)-basis of \(X\) which is uniformly bounded; that is, there exists \(M > 0\) such that \(\|P_n\| \leqslant M\) for all \(n \in \mathbb{N}\). Then \((P_n)_{n=1}^\infty\) admits a~subsequence which is a \(\pi\)-basis of \(X\).
\end{proposition}

\begin{proof}
Since \(X\) is separable, there exists a countable dense subset \(D = \{x_1, x_2, \dots\} \subseteq X\). For each \(k \in \mathbb{N}\) and \(\varepsilon > 0\), the filter convergence \(\lim_{n, \mathcal{F}} P_n x_k = x_k\) implies that the set
\[
A_{k, \varepsilon} = \{ n \in \mathbb{N} : \|P_n x_k - x_k\| < \varepsilon \}
\]
belongs to \(\mathcal{F}\).  Since \(\mathcal{F}\) is a filter, it is closed under finite intersections.  Thus, for each \(m \in \mathbb{N}\), the set
\[
B_m = \bigcap_{k=1}^m A_{k, \varepsilon/m}
\]
also lies in \(\mathcal{F}\).  Since every filter contains the Fréchet filter, each \(B_m\) is infinite.

We construct a subsequence \((n_m)_{m=1}^\infty\) inductively as follows:
\begin{itemize}
    \item Choose \(n_1 \in B_1\).
    \item For \(m \geq 2\), choose \(n_m \in B_m\) such that \(n_m > n_{m-1}\).
\end{itemize}
This construction is possible because each \(B_m\) is infinite.  By construction, for each \(m \in \mathbb{N}\) and for all \(k = 1, 2, \dots, m\), we have
\begin{equation}\label{eq:Pnm}
    \|P_{n_m} x_k - x_k\| < \frac{\varepsilon}{m}.
\end{equation}
Now, fix \(k \in \mathbb{N}\).  For all \(m \geq k\), equation \eqref{eq:Pnm} holds true. Since \(\varepsilon > 0\) was arbitrary, taking the limit as \(m \to \infty\) gives
\[
\lim_{m \to \infty} P_{n_m} x_k = x_k.
\]
This holds for all \(k \in \mathbb{N}\), so the subsequence \((P_{n_m})\) converges pointwise to the identity operator on the dense set \(D\). Now, let \(x \in X\) and \(\varepsilon > 0\) be given.  Since \(D\) is dense in \(X\), there is \(x_k \in D\) such that
\[
\|x - x_k\| < \frac{\varepsilon}{3(M + 1)}.
\]
Since \(P_{n_m} x_k \to x_k\), there exists \(M_k \in \mathbb{N}\) such that for all \(m \geq M_k\),
\[
\|P_{n_m} x_k - x_k\| < \frac{\varepsilon}{3}.
\]
Therefore, for \(m \geq M_k\), we have
\begin{align*}
\|P_{n_m} x - x\| &\leq \|P_{n_m} x - P_{n_m} x_k\| + \|P_{n_m} x_k - x_k\| + \|x_k - x\| \\
&\leq \|P_{n_m}\| \cdot \|x - x_k\| + \|P_{n_m} x_k - x_k\| + \|x_k - x\| \\
&\leq M \cdot \frac{\varepsilon}{3(M + 1)} + \frac{\varepsilon}{3} + \frac{\varepsilon}{3(M + 1)} \\
&< \frac{\varepsilon}{3} + \frac{\varepsilon}{3} + \frac{\varepsilon}{3} = \varepsilon.
\end{align*}
Thus, for each \(x \in X\), we have shown that  \(\lim_{m \to \infty} P_{n_m} x = x\).  Therefore, the subsequence \((P_{n_m})\) is a \(\pi\)-basis for \(X\).
\end{proof}

We shall require the following perturbation lemma (\cite[Lemma 2.4]{JRZ}).
\begin{lemma}\label{L:pert}
    Let $X$ be a Banach space and let $E\subseteq X$ be a finite-dimensional subspace. Let $T\colon X\to E$ be a surjective operator. If $F\subseteq X$ is a subspace with $k = \dim F \leqslant \dim E = n$ such that $\|T|_F - I_F\| < \varepsilon < 1$, where $\varepsilon\cdot k / (1-\varepsilon) < 1$, then
    \begin{enumerate}[i)]
        \item there is an operator $S$ from $X$ onto an $n$-dimensional subspace of $X$ such that $S|_F = I_F$, $\|S-T\| < \varepsilon\cdot k\cdot \|T\| / (1-\varepsilon) $, and $S^*[X^*] = T^*[X^*]$,
        \item if $T$ is a projection, $S$ may also be arranged to be such, in which case we have $\|S|_{T[X]} - I_{T[X]}\| < \varepsilon\cdot k / (1-\varepsilon)$.
    \end{enumerate}
\end{lemma}
\begin{corollary}\label{C:O poprawie operatorow}
Let \( X \) be a Banach space, and let \( P \colon X \to X \) be a finite-rank projection. Suppose \( D = \{x_i \colon i \in \mathbb{N}\} \) is a dense subset of \( X \). Define \( X_k \) as the linear span of \( \{x_1, \ldots, x_k\} \) for \( k \in \mathbb{N} \). Then, for every \( \varepsilon \in (0, 1) \), there exist a projection \( P^\prime \) and \( m \in \mathbb{N} \) such that:
\begin{itemize}
    \item \( \|P - P^\prime\| \leq \varepsilon \cdot k \cdot \|P\| \cdot (1 - \varepsilon)^{-1} \), where \( k = \dim \operatorname{im} P \),
    \item \( \operatorname{im} P^\prime \subseteq X_m \),
    \item \( \|P^\prime|_{\operatorname{im} P} - I_{\operatorname{im} P}\| < \varepsilon \cdot \dim \operatorname{im} P / (1 - \varepsilon) \).
\end{itemize}
\end{corollary}
For the sake of completeness, we sketch the proof of Corollary~\ref{C:O poprawie operatorow}.
\begin{proof}
     For each \( n \in \mathbb{N} \), the density of \( D \) ensures that there exists \( m \) sufficiently large to find a subspace \( F \subset X_m \) such that \( \|P_n|_F - I_F\| \) is sufficiently small and \( \dim F = \dim \operatorname{im} P_n \). 
    
    To construct \( F \), fix a normed basis \( \{e_1, \ldots, e_n\} \) of \( \operatorname{im} P \). Then, for each \( j \in \{1, \ldots, n\} \), choose \( x_{i_j} \) sufficiently close to \( e_j \) and set \( F = \operatorname{span}\{x_{i_j} : j \in \{1, \ldots, n\}\} \). 
    
    Next, apply Lemma \ref{L:pert} with \( E = \operatorname{im} P \), \( T = P \), and the subspace \( F \) as constructed above. The resulting operator \( S \) from the lemma is the desired \( P^\prime \).
\end{proof}

\begin{corollary}\label{C:corrbase}
    Let \( X \) be a Banach space with the \( \mathcal{F} \)-\( \pi \)-property. Suppose \( D = \{x_i \colon i \in \mathbb{N}\} \) is a dense subset of \( X \). Define \( X_k \) as the linear span of \( \{x_1, \ldots, x_k\} \) for \( k \in \mathbb{N} \). Then, there exists a sequence of projections \( (P_n^\prime)_{n=1}^\infty \) that witnesses the \( \mathcal{F} \)-\( \pi \)-property of \( X \) such that for every \( n \in \mathbb{N} \), there exists \( m \in \mathbb{N} \) with \( \operatorname{im} P_n^\prime \subseteq X_m \).
\end{corollary}

\begin{proof}
    Fix any sequence of projections \( (P_n)_{n=1}^\infty \) witnessing the \( \mathcal{F} \)-\( \pi \)-property of \( X \). Let \( (\varepsilon_n)_{n=1}^\infty \) be a sequence such that 
    \[
    \frac{\varepsilon_n \cdot \dim \operatorname{im} P_n \cdot \|P_n\|}{1 - \varepsilon_n} < \frac{1}{n}, \quad n \in \mathbb{N}.
    \]
    For \( n \in \mathbb{N} \), apply Corollary \ref{C:O poprawie operatorow} to \( P_n \) and \( \varepsilon_n \) to obtain \( P_n^\prime \). Note that for \( x \in X \),
    \[
    \|P_n^\prime x - x\| \leq \|P_n^\prime x - P_n x\| + \|P_n x - x\|,
    \]
    and
    \[
    \|P_n^\prime x - P_n x\| \leq \frac{\varepsilon_n \cdot \dim \operatorname{im} P_n \cdot \|P_n\|}{1 - \varepsilon_n} \|x\|,
    \]
    so
    \[
    \|P_n^\prime x - x\| \leq \frac{\varepsilon_n \cdot \dim \operatorname{im} P_n \cdot \|P_n\|}{1 - \varepsilon_n} \|x\| + \|P_n x - x\|.
    \]
    Substituting the condition on \( \varepsilon_n \), we have
    \[
    \|P_n^\prime x - x\| \leq \frac{1}{n} \|x\| + \|P_n x - x\|.
    \]
    Hence, if \( P_n(x) \underset{n, \mathcal{F}}{\longrightarrow} x \), it follows that \( P_n^\prime x \underset{n, \mathcal{F}}{\longrightarrow} x \).
\end{proof}
Moreover, we require a refinement of \cite[Proposition 3.6]{CasazzaKalton}.

\begin{lemma}\label{lem:cas}
    Let \( X \) be a Banach space, $\varepsilon>0$, and let \( T \colon X \to X \) be a linear operator. There exists \(0<\theta < \frac{1}{4}\), such that if $ \|T - T^2\|< \theta$, then there exists a projection \( P \colon X \to X \) such that 
    \begin{enumerate}[label=\roman*)]
    \item \(\{x \in X \colon Tx = x\} \subseteq P[X] \subseteq T[X]
    \), 
    \item \(\|Px-x\|<\varepsilon\) whenever \(\|Tx-x\|<\frac{\varepsilon}{2}\) and \(x\in B_X\)
    \item {}
    \[
        \|P\| \leqslant \frac{1}{2} \left( 1 + \frac{1 + 2\|T\|}{\sqrt{1 - 4\theta}} \right).
    \]
    \item\label{eq:PTestimate} 
    \[
        \|P-T\|\leq \frac{1}{2}(2\|T\|+1)\cdot \left(\frac{1}{\sqrt{1-4\theta}} - 1 \right).
    \]
    \end{enumerate}
    Moreover, $\theta$ may depend only on $\varepsilon$ and upper bound $\lambda \geq\|T\|$, regardless of $T$, and such a~dependence may be arranged to be continuous.
\end{lemma}

For our purposes, we require a slightly stronger statement, which can be derived using the same proof. Thus, we provide the full reasoning here for the sake of completeness. It seems to us that this stronger statement was already used in \cite[Theorem 3.7]{CasazzaKalton} and \cite{GhaFund}, since in either case it is not mentioned why corrected operators have non-trivial fixed points.
\begin{proof}
Define
\[
    S = \sum_{m=0}^\infty \binom{2m}{m} (T - T^2)^m\quad\text{ and }\quad P = \frac{1}{2}\big(I - (I - 2T)S\big).
\]
The function \((1 - 4z)^{-\frac{1}{2}}\) admits a power series expansion given by
\begin{equation}\label{eq:Snorm}
    \sum_{m=0}^\infty \binom{2m}{m} z^m \quad(|z| < \frac{1}{4}).
\end{equation}
From this, it follows that $\|S\| \leq (1 - 4\theta)^{-\frac{1}{2}}$. Moreover, observe that
\[
   (I - 2T)^2
   \;=\;
   I - 4\,T + 4\,T^2 
   \;=\;
   I \;-\; 4\,\bigl(T - T^2\bigr).
\]
On the other hand, by definition of $S$ we have, in the holomorphic functional-calculus sense,
\(
   S 
   \;=\; \bigl(1 - 4\,(T - T^2)\bigr)^{-\tfrac12}.
\)
Hence
\(
   S^2 
   \;=\; \bigl(1 - 4\,(T - T^2)\bigr)^{-1}.
\) 
It follows directly that
\[
   (I - 2T)^2 \, S^2 
   \;=\;
   \bigl[I - 4\,(T - T^2)\bigr] 
   \;\cdot\; \bigl[1 - 4\,(T - T^2)\bigr]^{-1}
   \;=\;
   I.
\]
Consequently, \(P\) is a projection on \(X\), and the stated estimate for \(\|P\|\) follows. Indeed,
\[
   P^2 
   \;=\; \Bigl(\tfrac12\bigl[I - (I - 2T)\,S\bigr]\Bigr)^2 
   \;=\; \tfrac14\,\Bigl[I - (I - 2T)S \Bigr]^2. 
\]
As
\[
   \bigl[I - (I - 2T)\,S\bigr]^2 
   \;=\; I \;-\; 2\,(I - 2T)\,S \;+\; (I - 2T)^2\,S^2,
\]
we have
\[
   P^2 
   \;=\; \tfrac14\,\bigl[\,
         I \;-\; 2\,(I - 2T)\,S \;+\; (I - 2T)^2\,S^2
      \bigr].
\]
so
\[
   P^2 
   \;=\; \tfrac14\,\bigl[I + I - 2(I - 2T)\,S\bigr] 
   \;=\; \tfrac12 \bigl[I - (I - 2T)\,S\bigr] 
   \;=\; P.
\]
Additionally, note that \(P\) can be expressed as:
\[
    P = 3T^2 - 2T^3 + \frac{1}{2}(I - 2T) \sum_{m=2}^\infty \binom{2m}{m} (T - T^2)^m = T + \frac{1}{2}(2T-I)\sum_{m=1}^\infty \binom{2m}{m} (T - T^2)^m.
\]
The remaining properties of \(P\) can be derived straightforwardly; in particular, $\theta$ may be chosen by the last equality. Indeed, the condition 
\(
    \|Px-x\|\leq \varepsilon
\)
follows from 
\[
    \frac{1}{2}(1+\|T\|)\sum_{m=1}^\infty \binom{2m}{m}\theta^m<\frac{\varepsilon}{2},\quad \|x\|\leq 1,\quad \|Tx - x\|<\frac\varepsilon2.
\]
Finally, as \(P-T=\frac{1}{2}(2T-I)\sum_{m=1}^\infty \binom{2m}{m} (T - T^2)^m\) by \eqref{eq:Snorm}, estimate \eqref{eq:PTestimate} follows.
\end{proof}
Let us denote by $\Theta\colon(\lambda,\varepsilon)\mapsto \theta$ the continuous map provided by the above Lemma. 
\begin{corollary}\label{cor:Fpi-from-Tn}
Let $X$ be a Banach space, $\FF$ be a filter on $\mathbb{N}$, and let $\Theta_n \colon \lambda \mapsto \Theta(\lambda, \frac1n)$. Suppose $(\lambda_n)_{n=1}^\infty$ is a~sequence of real numbers all greater than $1$. 

Assume:
\begin{enumerate}
    \item $D := \{x_i : i \in \mathbb{N}\}$ is a dense subset of $X$.
    \item $(T_n)_{n=1}^\infty$ is a sequence of finite-rank operators on $X$ such that
    \[
    \|T_n\|\;\leqslant\;\lambda_n 
    \quad\text{and}\quad
    \|T_n - T_n^2\|\;\leqslant\;\Theta_n(\lambda_n) \quad\text{for all }n.
    \]
    \item For every $x \in B_X$ and every $\varepsilon>0$, there is a set $A \in \FF$ such that for each $n\in A$ one can find $i\in \mathbb{N}$ with
    \begin{equation}\label{eq:approx-dense}
        \|x - x_i\|\;\leqslant\; \frac{\varepsilon}{12\,\lambda_n}
        \quad\text{and}\quad
        \|x_i - T_n x_i\|\;\leqslant\;\frac\varepsilon4.
    \end{equation}    
\end{enumerate}
Then $X$ has the $\FF$-$\pi$ property.
\end{corollary}
\begin{proof}
    Let $(P_n)_{n=1}^\infty$ be the sequence of improved projections associated with $(T_n)_{n=1}^\infty$, provided by Lemma~\ref{lem:cas}. We shall prove that $(P_n)_{n=1}^\infty$ witnesses the $\mathcal{F}$-$\pi$-property of $X$. Let us denote $
    \theta_n=\Theta_n(\lambda_n)$. By construction, each $P_n$ is a finite-rank projection in $X$. It remains to verify that for every $x \in X$, the sequence $(P_n x)$ is $\mathcal{F}$-convergent to $x$. Since $\mathcal{F}$-convergence is determined on bounded sets, it suffices to establish this for $x \in B_X$.

    Fix $x \in B_X$ and $\varepsilon > 0$. Let $A \in \mathcal{F}$ be given by~\eqref{eq:approx-dense}. As $\mathcal{F}$ contains the Fréchet filter, we may assume that $\frac{1}{n} < \frac{\varepsilon}{2}$ for all $n \in A$. For each $n \in A$, there exists $i \in \mathbb{N}$ such that
    \[
        \|x - x_i\| \leq \frac{\varepsilon}{12 \lambda_n}
        \quad \text{and} \quad
        \|x_i - P_n x_i\| \leq \frac{\varepsilon}{2}.
    \]
    Consequently, for every $n \in A$, we estimate
    \[
        \|x - P_n x\| 
        \leq \|x - x_i\| + \|x_i - P_n x_i\| + \|P_n x_i - P_n x\|
        \leq (1 + \|P_n\|) \|x - x_i\| + \frac{\varepsilon}{2}.
    \]
    By the properties of $P_n$, we further bound
    \[
        \|x - P_n x\| 
        \leq \left(1 + \frac{1}{2} \cdot \frac{1 + 2\|T_n\|}{\sqrt{1 - 4\theta_n}} \right) \frac{\varepsilon}{12 \lambda_n} + \frac{\varepsilon}{2}.
    \]
    Using the estimate $\|T_n\| \leq \lambda_n$ and $\theta_n \leq \frac{4}{5}$, we obtain
    \[
        \|x - P_n x\| 
        \leq \left( \lambda_n + \frac{1}{2} \cdot \frac{\lambda_n + 2\lambda_n}{\sqrt{1 - \frac{4}{5}}} \right) \frac{\varepsilon}{12 \lambda_n} + \frac{\varepsilon}{2}
        \leq \frac{6\varepsilon}{12} + \frac{\varepsilon}{2} 
        = \varepsilon.
    \]
    This completes the proof.
\end{proof}

\section{Proof of Theorem~\ref{Th:A}}

In \cite{GhaFund}, Ghawadrah dealt with the usual convergence, which allowed her to provide several properties under the same conditions. However, as we are investigating just $\FF$ convergence rather than the usual one, we need to separate providing the existence of projections from providing their approximation properties. To be more precise, let us fix the linearly independent countable dense sequence $(x_n)$ in the separable Banach space $X$ and denote $X_k=\operatorname{span}\{x_1, \ldots, x_k \}$. 
Before formulating the main Lemma, we define 
\[
    S:= \left\{(\sigma_{n}) \in (\IR^{\N^2})^{\N}\colon \forall_{n}  \exists_{m} \forall_{j \geq m} \forall_{i} \sigma_{n}(i,j) = 0   \right\},
\]
which is a $\Pi^0_3$-subset (hence Borel) of the product space $(\IR^{\N^2})^{\N^2}$.

    Let us explain that the proper way of thinking of $S$ is that $n$ plays the role of picking $P_{n}$ and $i,j$ says we are considering the coefficient of the value of the constructed function in $x_i$ in the $x_j$ axis. Note that by \cite[Theorem 3.1]{Kechris}, $S$ may be viewed as a Polish space, so even though it appears rather complicated, quantifying over it is not a problem. 

Recall also that $\FF$-projections, even if continuous, need not be uniformly bounded; therefore, we require the existence of a sequence $(\lambda_n)_{n=1}^\infty$ of constants responsible for the continuity of respective projections.

By $c_{00}^\IQ$ we denote $\{(\alpha_n)\in c_{00}: \forall_n \alpha_n \in \IQ\}$. Let $\Theta_n\colon \lambda \mapsto \theta $ be as in Corollary \ref{cor:Fpi-from-Tn}.
\begin{lemma}\label{lem:main}
Let $X$ be a Banach space. Then $X$ has the $\FF$-$\pi$-property if and only if the following condition holds:
\begin{equation*}
\exists_{(\lambda_n)_{n=1}^\infty} \ \exists{(\sigma_{n})\in S} 
\end{equation*} 
\begin{equation}\label{L: ciaglosc}
\forall_{(\alpha_i)\in c_{00}^\IQ} \forall_{n\in \N} \left\| \sum_{i=1}^\infty \alpha_i \sum_{j=1}^\infty \sigma_{n}(i,j) x_j \right\| \leq \lambda_n \left\| \sum_{i=1}^\infty \alpha_i x_i \right\| 
\end{equation}

\begin{equation}\label{L: norm T-T^2}
\forall_{(\alpha_i)\in c_{00}^\IQ}\forall_{n\in\IN}\left\| \sum_{i=1}^\infty \alpha_i \sum_{j=1}^\infty \sigma_{n}(i,j) x_j - \sum_{i=1}^\infty \alpha_i\sum_{j=1}^\infty\sum_{t=1}^\infty  \sigma_{n}(i,j)\sigma_{n}(j,t) x_t \right\| \leq \Theta_n(\lambda_n) \left\|\sum_{i=1}^\infty \alpha_i x_i \right\| 
\end{equation}
\begin{equation}\label{L: baza}
    \forall_{x\in B_X} \forall_{r\in \IN} \exists_{A\in\FF} \forall_{n\in A}\exists_{i\in \IN} \| x-x_i\| \leq \frac1r\cdot\frac{1}{12\lambda_n} \wedge \left\|x_i - \sum_{j=1}^\infty \sigma_{n}(i,j) x_j\right\| < \frac{1}{4r} 
\end{equation}

\begin{proof}
"\(\Rightarrow\)" Assume that \((P_n)_{n=1}^\infty\) is an \(\mathcal{F}\)-\(\pi\)-base for \(X\). By Corollary \ref{C:corrbase}, we may assume that \(\operatorname{im}(P_n) \subset X_m\) for some \(m \in \mathbb{N}\). Set \(\lambda_n :=\big\lceil \|P_n\| \big\rceil + 1\) and define \(\ve{\sigma} \in S\) by the condition
\[
    P_nx_i = \sum_{j=1}^\infty \sigma_{n}(i,j) x_j.
\]
The inclusion \(\operatorname{im}(P_n) \subset X_m\) ensures that above sum is in fact finite. 

Verification of conditions \eqref{L: ciaglosc} and \ref{L: norm T-T^2} is straightforward (note that since $P_n$'s are projections, condition \ref{L: norm T-T^2} is even simpler). To verify condition \eqref{L: baza} let us fix $x\in B_X$ and $r\in \IN$. Since $(P_n)$ witnesses the $\FF$-$\pi$-property of $X$, there is $A\in \FF$ such that for every $n\in A$ $\|P_nx-x\|\leq \frac{1}{8r}$. Hence for $n \in A$ we may find $i\in \IN$ such that 
\[
    \|x-x_i\|\leq \min\left\{\frac{1}{8r(1+\lambda_n)},\frac{1}{12r\lambda_n}\right\},
\] and thus
\[
\|x_i-P_nx_i\|\leq \|x-x_i\|+\|x-P_nx\|+ \|P_nx-P_n x_i\| \leq (1+\lambda_n)\|x-x_i\|+ \frac{1}{8r} \leq \frac{1}{4r}.
\]

"\(\Leftarrow\)" Suppose the proper \((\lambda_n)\) and \(\sigma\)'s are provided by our formula. The first step in the construction is to define  
\begin{equation}
    y_{n}^i = \sum_{j=1}^\infty \sigma_{n}(i, j) x_j,
\end{equation}
and  
\begin{equation}
    z_{n}^i = \sum_{j=1}^\infty \sigma_{n}(i, j) y^i_{n} = \sum_{j=1}^\infty \sum_{t=1}^\infty \sigma_{n}(i, j) \sigma_{n}(j, t) x_t.
\end{equation}

By the definition of the set \(S\), the above sums are finite. Let \(m\) denote their length. Since the sequence \((x_N)\) is linearly independent and dense in \(X\), we can define operators \(T_n \colon X \to X_m\) by the rule \(T_n(x_i) = y_{n}^i\). By condition \eqref{L: ciaglosc}, it follows that \(\|T_n\| \leq \lambda_n\), and by condition \eqref{L: norm T-T^2}, the estimate \(\|T_n - T_n^2\| \leq \Theta_n(\lambda_n)\) holds, since \(T_n^2x_i=z^i_n\). 

Moreover, by \eqref{L: baza} for every $x \in B_X$ and every $\varepsilon>0$, there is a set $A \in \FF$ such that for each $n\in A$ one can find $i\in \mathbb{N}$ with
    \begin{equation}\label{eq:approx-dense-2}
        \|x - x_i\|\;\leqslant\; \frac{\varepsilon}{12\,\lambda_n}
        \quad\text{and}\quad
        \|x_i - T_n x_i\|\;\leqslant\;\frac{\varepsilon}{4},
    \end{equation}    

Hence we may apply Corollary \ref{cor:Fpi-from-Tn} to conclude that $X$ satisfies the $\FF$-$\pi$-property. 
\end{proof}
\end{lemma}
We are ready to prove Theorem \ref{Th:A}.
\begin{proof}
    Let $\{f_i:i\in \IN\}$ be provided by Corollary \ref{cor:borel_selections} and set
    \[
    \mathbb{X}= \mathsf{SB} \times \IN^\IN \times S.
    \]
    Moreover, let 
    \begin{align*}
\mathbb{X}_1 
&= \bigcap_{(\alpha_i) \in c_{00}^\mathbb{Q}} 
    \bigcap_{n \in \mathbb{N}} 
\\
&\quad \Bigg\{ (X, (\lambda_n), \sigma)\in\mathbb{X}\colon\ 
    \left\| \sum_{i=1}^\infty \alpha_i 
      \sum_{j=1}^\infty \sigma_{n}(i, j) f_j(X) 
    \right\| 
    \leq \lambda_n 
    \left\| \sum_{i=1}^\infty \alpha_i f_i(X) \right\| 
  \Bigg\}
\end{align*}
        \[
        \mathbb{X}_{2} =  \bigcap_{(\alpha_i) \in c_{00}^\IQ}\bigcap_{n\in\IN}\left\{
        \begin{array}{l}
        \left( X, (\lambda_n), \sigma \right)\in \mathbb{X} \colon \\
        \quad \left\|
        \sum_{i=1}^\infty \alpha_i \sum_{j=1}^\infty \sigma_{n}(i, j) f_j(X)\right.\\ 
        \qquad- 
        \left.\sum_{i=1}^\infty \alpha_i \sum_{j=1}^\infty \sum_{t=1}^\infty 
        \sigma_{n}(i, j) \sigma_{n}(j, t) f_t(X) 
        \right\| \\
        \quad \leq \Theta_n(\lambda_n) \left\| \sum_{i=1}^\infty \alpha_i f_i(X) \right\|
        \end{array}
        \right\}
    \]
    and for $r\in\IN$ set
        \begin{align*}
\mathbb{W}_{r}
&= 
    \bigcap_{n\in\IN}
     \bigcup_{i \in \mathbb{N}} 
    \Bigg\{ (X, (\lambda_n), \sigma, x,  A)\in \mathbb{X}\times B_{C(\Delta)}\times \{0,1\}^\IN \colon
\\
&\quad n \notin A \vee \left(
\| x-f_i(X)\| \leq \frac1r\cdot\frac{1}{12\lambda_n} \wedge \left\|f_i(X) - \sum_{j=1}^\infty \sigma_{n}(i,j) f_j(X)\right\| < \frac{1}{4r}\right)
  \Bigg\}.
\end{align*}
Note that if $(X, (\lambda_n), \sigma, x,  A)\in \mathbb{W}_r$, then $x\in X$.

Given a product space $A\times B$, by $\Pi_A\colon A\times B \to A$ we mean the projection onto the \(A\)-coordinate. Let 
\[
\mathbb{V}_r=\mathbb{W}_r \cap \mathbb{X}\times B_{C(\Delta)}\times \FF, \  \mathbb{V}=\bigcap_{r\in\IN}  \Pi_{\mathbb{X}\times B_{C(\Delta)}} [\mathbb{V}_r].
\]
Finally set 
\[
\mathbb{X}_3 = \mathbb{X} \setminus \Pi_{\mathbb{X}}\left[\left( \mathbb{X} \times B_{C(\Delta)}\right)\setminus \mathbb{V}\right].
\]

Note that by Lemma \ref{lem:main} $\mathsf{SB}_{\FF-\pi}=\Pi_{\mathsf{SB}}\left[\mathbb{X}_{1}\cap \mathbb{X}_{2} \cap \mathbb{X}_{3}\right]$. Moreover, if $\FF$ is $\Sigma_s^1$, then
clearly $\mathbb{X}_{1},\mathbb{X}_{2}, \mathbb{W}_{r}$ are Borel, $\mathbb{V}_{r}, \mathbb{V}$ are $\Sigma_s^1$, thus $\mathbb{X}_3$ is $\Pi^1_{s+1}$ and we conclude that $\mathsf{SB}_{\FF-\pi}$ is $\Sigma_{s+2}^1$.

In order to get the claim concerning countably generated filter \(\FF\) generated by sets \(\{F_l\colon l\in\IN\}\), one needs to provide analogous reasoning, yet in the condition \eqref{L: baza} of Lemma \ref{lem:main} projective quantifier $\exists_{A\in\FF}$ may be replaced by the countable one $\exists_{l\in\IN}$ with further referring to $A_l$ instead of $A$. In such a case we set 
            \begin{align*}
\mathbb{V'}
&= 
    \bigcap_{r\in\IN}
    \bigcup_{l\in\IN}
    \bigcap_{n\in\IN}
     \bigcup_{i \in \mathbb{N}} 
    \Bigg\{ (X, (\lambda_n), \sigma, x)\in \mathbb{X}\times B_{C(\Delta)}\times \{0,1\}^\IN \colon
\\
&\quad n \notin A_l \vee \left(
\| x-f_i(X)\| \leq \frac1r\cdot\frac{1}{12\lambda_n} \wedge \left\|f_i(X) - \sum_{j=1}^\infty \sigma_{n}(i,j) f_j(X)\right\| < \frac{1}{4r}\right)
  \Bigg\}
\end{align*}
along with
\[
\mathbb{X'}_3 = \mathbb{X} \setminus \Pi_{\mathbb{X}}\left[\left( \mathbb{X} \times B_{C(\Delta)}\right)\setminus \mathbb{V'}\right].
\]
In such a case sets  $\mathbb{X}_{1},\mathbb{X}_{2}, \mathbb{V'}$ are Borel, yet $\mathbb{X'}_{3}$ is $\Pi_1^1$, so $\mathsf{SB}_{\FF-\pi}$ is $\Sigma^1_2$.

  \end{proof}
Although we were unable to reduce the problem to a genuine $\Sigma^1_3$ set, we conjecture that the formula above is not substantially simpler. Achieving such a reduction would likely require deriving convergence on the entire space $X$ from convergence on a dense subset, or encoding $\mathcal{F}$-convergence using countable quantifiers. A significant obstacle to this approach is the potential failure of uniform boundedness of the norms of the witnessing projections. This consideration also motivates the formulation of our condition, which begins with the existence of the sequence $(\lambda_n)_{n=1}^\infty$. This existential quantification already places the complexity beyond the Borel hierarchy and obviates the need to enforce rationality of the $\sigma$'s. Indeed, even restricting quantification to a countable set $S$ would not suffice to reduce the resulting complexity.


\begin{remark}
    One might ask whether our result justifies the effort involved, as the problem might initially appear straightforward. A natural first attempt could be to express the property as the formula:
    \[
        \exists_{(P_n)} \forall_{x \in X} \forall_{\varepsilon > 0} \exists_{A \in \mathcal{F}} \forall_{n \in A} \|x - P_n x\| < \varepsilon,
    \]
    which, at first glance, resembles the definition of a \(\Sigma^1_3\) set (possibly modulo some technical details). However, this is not the case. The initial existential quantifier requires the existence of a sequence of projections, but the space over which these projections are considered is non-trivial to define. 

    One might initially consider the space of all bounded linear operators \(B(X, X)\). However, this space is typically not separable in the norm topology, and its unit ball is Polish in the Strong Operator Topology (SOT) only if \(X^*\) is separable. A more refined approach would restrict attention to the space of finite-rank operators, which is separable (in either norm or SOT) only if \(X^*\) is separable. Consequently, directly deducing the final projective complexity class from such a formula appears infeasible without employing the detailed approximation arguments presented earlier in this work.
\end{remark}

\section{Proof of Theorem~\ref{Th:B} and Theorem~\ref{Th:C}}
This section is heavily inspired by the proof of \cite[Theorem B]{RKS}.
\begin{proof}[Proof of Theorem~\ref{Th:B}]
Set 
\[
    \mc{A}:= \Bigg\{A \subset \IN\colon \exists_{x \in X} \exists_{\varepsilon>0}\; A \supset \Big\{ n \in \IN\colon \big\|P_nx - x \big\| \leq \varepsilon \Big\}\Bigg\}.
\]
Clearly $\mc{A} \subset \FF$. Note that $(P_n)$ witnesses that $X$ has $\FF^\prime$-$\pi$-property for any filter $\FF'$ on $\IN$ such that $\mc{A} \subset \FF' $. Now, for every $n \in \N$, consider the set
\[
\mc{B}_n= \Big\{ (x,\varepsilon,A)\in X \times (0,\infty)\times \{0, 1\}^\IN\colon n \in A \vee \big\|P_nx - x \big\| > \varepsilon \Big\}.
\]
By the continuity of $P_n$ we obtain that $\mc{B}_n$ is open. Observing that
\[
\mc{A} = \on{proj}_{\{0, 1\}^\IN}\Big[\bigcap_{n \in \IN} \mc{B}_n\Big],
\]
we deduce that $\mc{A}$ is analytic.\smallskip

As $\mc{A} \subset \FF$, finite intersections on elements of $\mc{A}$ are readily non-empty. Consequently, $\mc{A}$ generates a filter $\FF' \subset \FF$. We have:
\[
    \begin{aligned}
    \FF^\prime= &\{A \supset \IN\colon \exists_{n \in \IN} \exists_{A_1, A_2, \ldots, A_n \in \mc{A}} A \supset A_1 \cap A_2 \cap \ldots \cap A_n \} \\
    = & \bigcup_{n \in \IN} \on{proj}_{(\{0, 1\}^\N)_1} [\{ (A, A_1, \ldots, A_n)\in (\{0, 1\}^\N)^{n+1}\colon \\
    & A \supset A_1 \cap \ldots \cap A_n \wedge A_1, \ldots, A_n \in \mc{A}\}],
\end{aligned}
\]
where $\on{proj}_{(\{0, 1\}^\N)_1}$ is the projection onto the first coordinate. Consequently, $\FF'$ is an analytic filter and $(P_n)_{n=1}^\infty$ witnesses that $X$ has $\FF'$-$\pi$-property since $\mc{A}\subset \FF' \subset \FF$.
\end{proof}
\begin{proof}[Proof of Theorem~\ref{Th:C}]
  We start with describing the family of all analytic filters on $\IN$. Recall that set $B$ is analytic within Polish space $Y$ if and only if there exists a closed subset $D\subset Y\times \IN^\IN$ such that $B=\Pi_Y[D]$, while the space $F(Y\times \IN^\IN)$ is a standard Borel space while equipped with the Effros-Borel structure. We may consider the following sets
    \begin{align*}
    \mathbb{INT}
    &=
    \bigg\{D\in F\left(\{0,1\}^\IN\times\IN^\IN\right)\colon    \forall_{(a_1,b_1),(a_2,b_2)\in \{0,1\}^\IN\times\IN^\IN} \exists_{(a_3,b_3)\in \{0,1\}^\IN\times\IN^\IN}
    \\
    &
    \qquad (a_1,b_1)\notin D \vee (a_2,b_2)\notin D \vee \left[ a_3=a_1\cap a_2 \wedge (a_3,b_3)\in D\right] \bigg\},    
    \end{align*}
     \begin{align*}
    \mathbb{SUP}
    &=
    \bigg\{D\in F\left(\{0,1\}^\IN\times\IN^\IN\right)\colon \forall_{(a_1,b_1)\in \{0,1\}^\IN\times\IN^\IN}\forall_{a\in\{0,1\}^\IN} \exists_{(a_2,b_2)\in \{0,1\}^\IN\times\IN^\IN}
   \\
    &\qquad (a_1,b_1)\notin D \vee a\not\supset a_1 \vee \left[a=a_2 \wedge (a_2,b_2) \in D\right]    \bigg\}    
    \end{align*}

Note that $\mathbb{INT}$ stands for family of those $D$ whose projections are closed with respect to intersections of its elements, while $\mathbb{SUP}$ stands for family of those $D$ whose projections are closed with respect to supersets of its elements. We may hence define 
\[\mathbb{AF}=\mathbb{SUP}\cap\mathbb{INT}\]
and observe that $\FF\subset\{0,1\}^\IN$ is an analytic filter if and only if there is a $D\in \mathbb{AF}$ such that $\FF=\Pi_{\{0,1\}^\IN}[D]$.

Moreover, it is straightforward to verify that both $\mathbb{SUP},\mathbb{INT}$ are $\Pi^1_2$, since for any Polish $Y$ and $y\in Y$ set $\{D\in F(Y)\colon y\in D\}=\bigcap_{n\in\IN} \{D\in F(Y)\colon D\cap B(y,\frac1n)\}$ is Borel within $F(Y)$. 

Let $\mathbb{X}, \mathbb{X}_1, \mathbb{X}_2$ be as in the proof of Theorem \ref{Th:A}, and set 
$    \mathbb{Y}=\mathbb{X}\times F\left(\{0,1\}^\IN\times\IN^\IN\right), \mathbb{Y}_1=\mathbb{X}_1\times F\left(\{0,1\}^\IN\times\IN^\IN\right), \mathbb{Y}_2=\mathbb{X}_2\times F\left(\{0,1\}^\IN\times\IN^\IN\right)$.
Moreover, for $r\in\IN$ set
        \begin{align*}
\mathbb{U}_{r}
&= 
    \bigcap_{n\in\IN}
     \bigcup_{i \in \mathbb{N}} 
    \Bigg\{ (X, (\lambda_n), \sigma, D, x,  A)\in \mathbb{Y}\times B_{C(\Delta)}\times \{0,1\}^\IN  \colon A \in \Pi_{\{0,1\}^\IN}[D] \wedge 
\\
&\quad \left[ n\notin A \vee \left(
\| x-f_i(X)\| \leq \frac1r\cdot\frac{1}{12\lambda_n} \wedge \left\|f_i(X) - \sum_{j=1}^\infty \sigma_{n}(i,j) f_j(X)\right\| < \frac{1}{4r} \right)\right]
  \Bigg\}.
\end{align*}
Note that 
\[
A \in \Pi_{\{0,1\}^\IN}[D] \Leftrightarrow \exists_{b\in\IN^\IN} (A,b)\in D,
\]
so $\mathbb{U}_{r}$ is $\Sigma_1^1$.
Now let
\[
\mathbb{U}=\bigcap_{r\in\IN}  \Pi_{\mathbb{Y}\times B_{C(\Delta)}} [\mathbb{U}_r] \mbox{ and } \mathbb{Y}_3 = \mathbb{Y} \setminus \Pi_{\mathbb{Y}}\left[\left( \mathbb{Y} \times B_{C(\Delta)}\right)\setminus \mathbb{U}\right].
\]
Note that by Lemma \ref{lem:main} and Theorem \ref{Th:B} 
\[
\mathsf{SB}_{\text{Filter-}\pi}=\Pi_{\mathsf{SB}}\left[\mathbb{Y}_{1}\cap \mathbb{Y}_{2} \cap \mathbb{Y}_{3}\cap\left(\mathbb{X}\times\mathbb{AF}\right)\right].
\]
Observe that $\mathbb{Y}_{1},\mathbb{Y}_{2}$ are Borel, $\mathbb{U}$ is $\Sigma_1^1$, thus $\mathbb{Y}_3$ is $\Pi^1_{2}$. Set $\mathbb{X}\times\mathbb{AF}$ is also $\Pi^1_2$, so we conclude that $\mathsf{SB}_{\text{Filter-}\pi}$ is $\Sigma_{3}^1$.
\end{proof}

\section{Problems}
\begin{question}
In the analysis of the descriptive complexity of the set $\mathsf{SB}_{{\rm Fr}\text{-}\pi}$ by Gha\-wa\-drah \cite{GhaFund}, a choice of dense sequences in a separable Banach space is employed to characterise the $\pi$-property with respect to the Fréchet filter. It remains unclear whether her approach presupposes the linear independence of this sequence, however this seems necessary. Furthermore, the claimed Borel complexity of $\Sigma^0_6$ for $\mathsf{SB}_{{\rm Fr}\text{-}\pi}$ appears to hinge on the definition of certain operators $T_{K,\varepsilon}$, whose well-posedness may be inadequately justified. Let us then state the following problem formally: 
\begin{quote}\emph{is there a continuous (with respect to an admissible topology) way of picking dense and linearly independent subsets of separable Banach spaces?}
\end{quote}
\end{question}
\bibliographystyle{plain}

\begin{thebibliography}{}
\bibitem{AK} F. Albiac, N.~J. Kalton, \emph{Topics in Banach Space Theory}. Graduate Texts in Mathematics, 233. New York: Springer-Verlag, 2006.
\bibitem{avilesetal} A. Avil\'es, V. Kadets, A. P\'erez, S. Solecki, Baire theorem for ideals of sets, \emph{J. Math. Anal. Appl.} \textbf{445} (2) (2017) 1221--1231.
\bibitem{Casazza} P.~G. Casazza. \emph{Approximation properties}. In: Handbook of the geometry of Banach spaces, Vol. I. North-Holland, Amsterdam, 2001, pp. 271--316.
\bibitem{CasazzaKalton} P.~G.~Casazza, N.~J.~Kalton, Notes on approximation properties in separable Banach spaces. \emph{Lect. Notes Lond. Math. Soc.} \textbf{158} (1991), 49--65. 
\bibitem{Cuth1} M. C\'{u}th, M. Dole\v{z}al, M. Doucha, and O. Kurka, Polish spaces of Banach spaces, \emph{Forum of Mathematics, Sigma} \textbf{10} (2022), e26.
\bibitem{Cuth2} M. C\'{u}th, M. Dole\v{z}al, M. Doucha, and O. Kurka, Polish spaces of Banach spaces. Complexity of isometry and isomorphism classes, \emph{J. Inst. Math. Jussieu} \textbf{23}(4) (2024), 1919--1957.
\bibitem{RKS} N.~de~Rancourt, T.~Kania and J.~Swaczyna, Continuity of coordinate functionals of filter bases in Banach spaces. \emph{J. Funct. Anal.}. \textbf{284} (2023), no. 9, 109869.
\bibitem{GhaFund} G. Ghawadrah, The descriptive complexity of approximation properties in an admissible topology. \emph{Fund. Math.} \textbf{249} (2020), no. 3, 303–309.
\bibitem{GhaHous} G. Ghawadrah, The descriptive complexity of the family of Banach spaces with the bounded approximation property. \emph{Houston J. Math.} \textbf{43} (2017), no. 2, 395–401.
\bibitem{GhaAr} G. Ghawadrah, The descriptive complexity of the family of Banach spaces with the $\pi$-property. \emph{Arab. J. Math.} \textbf{4} (2015), no. 1, 35--39.
\bibitem{GoSR} G.~Godefroy and J.~Saint-Raymond, Descriptive complexity of some isomorphism classes of Banach spaces, \emph{J. Funct. Anal.}, \textbf{275} (2018), 1008--1022.
\bibitem{HajekMedina} 
P.~H\'ajek and R.~Medina, 
Compact retractions and Schauder decompositions in Banach spaces, 
\emph{Trans. Amer. Math. Soc.} \textbf{376} (2023), 1343--1372. 
\bibitem{JRZ}  W.B.~Johnson, H.P.~Rosenthal, and M.~Zippin, On bases, finite dimensional decompositions and weaker structures in Banach spaces. \emph{Israel J. Math.} \textbf{9} (1971), 488--506.
\bibitem{KaniaSwaczyna} T.~Kania, J.~Swaczyna, Large cardinals and continuity of coordinate functionals of filter bases in Banach spaces. \emph{Bull. Lond. Math. Soc}. \textbf{53} (2021), no. 1, 231--239.
\bibitem{Kechris} A.~S.~Kechris, \emph{Classical descriptive set theory}, vol. 156 of Graduate Texts in Mathematics, Springer-Verlag, New York, 1995.
\bibitem{Kochanek} T.~Kochanek, $\mathcal{F}$-bases with brackets and with individual brackets in Banach spaces, \emph{Studia Math.} \textbf{211} (2012), 259--268.

\end{thebibliography}

\end{document}